%% file: main_siam.tex
\documentclass[hidelinks,onefignum,onetabnum]{siamart250211}

\input{ex_shared}

\ifpdf
\hypersetup{
  pdftitle={Optimal control of nonlocal conservation laws and the singular limit problem},
  pdfauthor={A. Keimer, L. Pflug, and J. Rodestock}
}
\fi

\allowdisplaybreaks[4]
\begin{document}

\title{Optimal Control Problems with Nonlocal Conservation Laws: Existence of Optimizers and Singular Limits in Approximations of Local Conservation Laws}

\maketitle
\renewcommand*\descriptionlabel[1]{\hspace\labelsep\textit{#1}}
\renewcommand{\textbf}[1]{\textit{#1}}
\begin{abstract}
This contribution considers optimal control problems subject to nonlocal conservation laws -- those in which the velocity depends nonlocally (i.e., via a convolution) on the solution -- and the so-called singular limit. First, the existence of minimizers is demonstrated for a broad class of optimal control problems, involving optimization over the initial datum, velocity, and nonlocal kernel for classical tracking-type \(L^{2}\) cost functionals.
Then, it is proven that the obtained minimizers converge to minimizers of the corresponding local optimal control problem when the kernel function of the convolution is of exponential type and approaches a Dirac distribution. Finally, some numerical results are presented.
\end{abstract}

\begin{keywords}
nonlocal conservation laws, optimal control, singular limit, optimal control of conservation laws
\end{keywords}

\begin{MSCcodes}
49J20, 49J22, 35L65
\end{MSCcodes}

\tableofcontents
\listoffigures

\section{Introduction}
In the last ten years, the theory of the nonlocal scalar conservation law
\begin{equation}\label{eq:con}
\partial_t q_\eta(t, x) + \partial_x\bigg(V\bigg( \tfrac{1}{\eta}\int^\infty_x \!\! \gamma\left(\tfrac{y-x}{\eta}\right) q_\eta(t, y) \dd y\bigg)q_\eta(t, x)\bigg) = 0 \text{ in }(0, T) \times \R,
\end{equation}
for $\eta > 0$, has been studied extensively~\cite{keimernonlocalbalance2017,Keimer2018,Chiarello2018,Amorim2015,Chiarello2018-yp,CocliteBV2022} (a nonexhaustive list). Existence and uniqueness without an entropy condition for~\cref{eq:con} have been established, along with a (semi-)explicit solution formula making use of the method of characteristics~\cite{keimernonlocalbalance2017,Keimer2018}. Moreover, it is known that, if the initial datum is continuous, the solution remains continuous~\cite{keimernonlocalbalance2017}. This is a pleasant regularity result, that does not hold for ``standard'' hyperbolic conservation laws~\cite[Section 3.4]{evanspartial}.\\
Applications of this and related equations have also been considered, ranging from nanotechnology~\cite{Pflug2020-mb,Spinola2020-nd} to traffic~\cite{Chiarello2024-od}. In these works, $V'\leq 0$ is generally assumed~\cite{Greenberg1959,Richards1956}. The latter is an assumption that we make in this work as well.
However, controllability and optimal control of nonlocal conservation laws have only been studied sparsely for some special cases~\cite{Grschel2014,Colombo2010,Spinola2020-nd,spinolaphdthesis}, including boundary controllability~\cite{Bayen2021}.

Recently, advances have been made~\cite{herty2025} for a setup rather close to the problem class we consider in this article, including a well-posedness result.
In the present work, we want to extend said result to the optimal control problem
\begin{align}
\inf_{\substack{q_0 \in \mQ \\ V \in \mV \\ \gamma \in \Gamma}} & \displaystyle J\big(q[q_0, V], W[q[q_0, V, \gamma]], q_0, V, \gamma\big)  \label{eq:op1}\\
    \text{ s.t. $q[q_0, V, \gamma]$}&\text{ is the unique solution to~\cref{eq:con}} \label{eq:op2}\\
    q(0, \cdot) &\equiv q_0,  &&\text{ on } \R  \label{eq:op3},
\end{align}
for a wide class of objective functions $J$ with fewer assumptions on the admissible set than those in \cite{herty2025} (we do not require \(TV\) bounds on the initial data), optimizing in $q_0$, $\gamma$, and $V$ all at once rather than in $q_0$ only. That is, we are interested in the optimal control problem with controls also in the ``coefficients'' of the PDE. We find, for example, that for existence of an optimal control, it is enough for $q_0$ to be uniformly essentially bounded and in a weak\(*\) compact subset of $L^2(\R)$. Additionally, we provide numerical results for the optimization in $q_0$ only, pursuing a \emph{discretize-then-differentiate approach}~\cite[section 4]{Giles2003} to obtain a derivative.

These optimal control problems are interesting both because there are few results on them and because the nonlocal conservation law is connected to the well-known local conservation law ($T>0$)
\begin{align}
\partial_t q + \partial_xf(q) &= 0,&& \text{on } (0, T) \times \R, \label{eq:oco}
\intertext{ supplemented by an initial condition}
q(0,\cdot)&\equiv q_{0}&& \text{on } \R
\end{align}
for $q_0 \in L^1_\text{loc}(\mathbb{R})$. Well-posedness, that is, existence and uniqueness of the entropy solution~\cite{Kruzkov1970,Bressan2000,godlewski1991}, is a widely celebrated result. The relation between the local equation~\cref{eq:oco} and~\cref{eq:con} is described by the \emph{singular limit problem}, which was first raised in~\cite{Amorim2015}. This is the question whether, for fixed $q_0 \in L^\infty(\R, \R_{\geq 0})$, one can ensure that the limit
\begin{equation}\label{eq:sl}
\lim_{\eta \rightarrow 0 } q_{\eta}[q_0] = q[q_0] \text{ in }L^1_{\text{loc}}\big((0, T) \times \R\big),
\end{equation}
where $q[q_0]$ is the weak entropy solution of the ``local'' conservation law
\begin{align}
    \partial_t q(t, x) + \partial_x\big(q(t, x)V(q(t, x))\big) &= 0,&& \text{on } (0, T) \times \R,  \label{eq:ocov} \\
q(0,\cdot)&\equiv q_{0}&& \text{on } \R \notag
\end{align}
with initial datum $q_0$ and $q_\eta[q_0]$, solves the nonlocal equation~\cref{eq:con} with initial datum $q_0$.

For~\cref{eq:sl} to hold, we must either make assumptions on the derivatives of $V$, if we are working with an Ole\u{\i}nik-type convergence result from nonlocal to local \cite{keimeroleinik,Crippa2021}, or impose total variation bounds on $q_0$ \cite{CocliteSL2022,Keimer2025,Colombo2023,colombo2023overview,Colombo2019,Keimer2023disc,Friedrich2024,keimersign2023}. However, as the Ole\u{\i}nik compactness and convergence result in the singular limit is quite restrictive and imposes further assumptions on the initial datum and derivatives of \(V\), we will only consider the singular limit convergence in the case of \(TV\) bounds.

The naturally ensuing question is whether the optimal control problem of~\cref{eq:con} is indeed an approximation of
\begin{align}
\inf_{\substack{q_0 \in \mQ \\ V \in \mV}} & \displaystyle J\big(q[q_0, V], q[q_0, V], q_0, V, \gamma\big)  \label{eq:opp1}\\
    \text{ s.t. $q[q_0, V]$}&\text{ is the unique entropy solution to~\cref{eq:ocov}} \label{eq:opp2}\\
    q(0, \cdot) &\equiv q_0,  &&\text{ on } \R  \label{eq:opp3},
\end{align}
in the sense that if $q_{0, \eta}^\ast$ (or $V_\eta^\ast$) is an optimizer of the optimal control problem
\begin{align}
\inf_{\substack{q_{0, \eta} \in \mQ \\ V \in \mV}} & \displaystyle J\big(q_\eta[q_{0, \eta}, V], W[q_{0, \eta}, V], q_{0, \eta}, V, \gamma\big)  \label{eq:o1}\\
    \text{ s.t. $q_\eta[q_{0, \eta}, V]$}&\text{ is the unique solution to~\cref{eq:con}} \label{eq:o2}\\
    q_\eta(0, \cdot) &\equiv q_{0, \eta},  &&\text{ on } \R  \label{eq:o3},
\end{align}
then every accumulation point of the sequence $\big(q_{0, \eta}^\ast \big)_{\eta >0}$ (or $\big(V_{\eta}^\ast \big)_{\eta >0}$) is an optimizer of~\cref{eq:opp1,eq:opp2,eq:opp3}. The answer to this is largely positive, as shown in~\cite{herty2025}, and an extension of that result to optimization in $q_0$ and $V$ simultaneously, in a broader, less regular setting, can be found in \cref{sec:singular_limit} of this paper. We will also be able to observe this convergence numerically.
This yields a means of approximating~\cref{eq:opp1,eq:opp2,eq:opp3}, which is itself a well-known problem.
Numerical methods, sensitivity analysis, optimality conditions (even for state constrained optimal control problems), and applications of optimal control problems have also been extensively studied~\cite{Ulbrich2003,Schmitt2021,LeVeque1992,Liard2022,Ulbrich2002,herty2025,Ulbrich1999, Adimurthi2014,Hajian2017,2Adimurthi2014,Li2023}, including flux identification~\cite{Castro2011,James1999} and special cases like Burgers' equation~\cite{Castro2009,James1999,CASTRO2008}.

\section{Well-posedness of optimal control with respect to \texorpdfstring{$V$, $\gamma$, $q_0$}{V, γ, q0} in a broad setting}
In this section, we study the existence of the corresponding nonlocal optimal control problem, which we formulate below.

\begin{definition}[the optimal control problem considered]\label{defi:optimal_control_problem}

For $T \in \R_{>0}$, we consider the problem
\begin{equation}\label{Eq:O}
\begin{aligned}
    \inf_{\substack{q_0 \in \mQ \\ V \in \mV \\ \gamma \in \Gamma}} & \displaystyle J\big(q[q_0, V], W[q[q_0, V, \gamma]], q_0, V, \gamma\big)\\
    \text{ s.t. $q[q_0, V, \gamma]$}&\text{ is the weak solution to}\\
    \partial_t q + \partial_x \left(q V\big(W[q]\big) \right) &= 0, &&\text{ on } (0,T)\times\R,\\
    q(0, \cdot) &\equiv q_0,  &&\text{ on } \R,
\end{aligned}
\end{equation}
where
\begin{equation}
W[q](t, x) \coloneqq \int_x^\infty \gamma\left(y-x \right)q(t, y)\dd y \quad \forall (t, x) \in (0, T) \times \R,\label{eq:nonlocal_operator}
\end{equation}
 and $\mQ \subseteq L^{\infty}(\R)$, \(q_{\max}\coloneqq \|q_{0}\|_{L^{\infty}(\R)}\),\ $\mV \subseteq C\big([0, q_{\max}], \R_{\geq 0}\big)$, and $\Gamma \subseteq L^1\big(\R_{\geq 0}, \R_{\geq 0}\big)$ are properly chosen admissible control sets. The objective functional $J$ will be specified later.
\end{definition}
For such a problem to be well-posed, we require further assumptions on the objective function \(J\) as well as the sets \(\Gamma,\mQ,\) and \(\mV\).
\begin{assumption}[\(\mQ,\mV,\Gamma\), and \(J\)]\label{ass:optimal_control}
    We require that there exist nonnegative constants \(\mC_{\Gamma},\mC_{\mV},\ \bar{R}, q_{\max},\ \eps\in\R_{>0}\) such that
    \begin{enumerate}[leftmargin=15pt,label=(\alph*)]
    \item \(\mV \coloneqq \Big \lbrace V \in W^{1, \infty}\big([0, q_{\max}]; \R_{\geq 0}\big): V'\leqq 0,\ \lVert V \rVert_{W^{1, \infty}([0, q_{\max}])} \leq \mC_{\mV} \Big \rbrace\),
    
    \item \(\Gamma \coloneqq  \Big \lbrace \gamma \in L^1(\R_{\geq 0}, \R_{\geq 0}): \gamma\text{ is nonincreasing a.e., }\lVert \gamma \rVert_{L^\infty(\R_{\geq 0})} \leq \mC_{\Gamma}, \\
     \phantom{\qquad\qquad}\|\gamma\|_{L^{1}(\R_{\geq0})}=1,\ \gamma(x) \leq x^{-(1+\eps)} \text{ for a.e. }x \in \R_{\bar R>0} \Big \rbrace \), and
    \item at least one of the following three assumptions holds:
    \begin{enumerate}[leftmargin=20pt,label=(\arabic*)]
        \item \textbf{General requirements of existence of optimizers:}\\ \label{item:1}
 Let \(X\coloneqq C\big([0,T]; L^2(\R)\big) \times C\big([0,T]; H^1(\R)\big) \times L^2(\R) \times C\big([0, q_{\max}], \R\big) \times L^1(\R).\)
       Then we assume that if we have a sequence 
        \[\big(q_k, w_k, q_{0, k}, V_{k}, \gamma_k\big)_{k \in \N} \subseteq X,\ \text{and } (q^*, w^*, q_0^*, V^*, \gamma)\in X\] such that
        \begin{itemize}
        \item $q_k \rightarrow q^*$ uniformly in first and $L^2(\R)$-weakly in second component
        \item $w_k \rightarrow w^*$ uniformly in first and $H^1(\R)$-weakly in second component
        \item $q_{0, k} \rightarrow q_0^*$ in $L^2(\R)$-weakly
        \item $V_k \rightarrow V^*$ strongly in $C\big([0, q_{\max}], \R_{\geq 0} \big)$
        \item $\gamma_k \rightarrow \gamma^*$ strongly in $L^1(\R)$
        \end{itemize}
        we have
        \begin{equation}\label{eq:lsem}
        \liminf_{k \rightarrow \infty} J(q_k, w_k, q_{0, k}, V_k, \gamma_k) \geq J(q^*, w^*, q_{0}^*, V^*, \gamma^*).
        \end{equation}
        Furthermore, \(J\) satisfies, for a given $s \in [0, T]$ and $K\in\R_{>0}$,
        \begin{equation}\label{eq:coer}
            J(q, w, q_0, V, \gamma) \geq K\| q(s, \cdot) \|_{L^2(\R)}-K\ \forall (q, w, q_0, V, \gamma) \in X.
        \end{equation}
        Moreover,
        \[
        \mQ \coloneqq \left \lbrace q_0 \in L^2(\R): 0 \leq q_0 \leq q_{\max} \text{ a.e.} \right \rbrace.
        \]
        \item \textbf{Strong convergence of the sequence of minimizers with respect to the nonlocal reach \(\eta\in\R_{>0}\) and singular limit for bounded \(\Omega\):}\\ There is a $\mC_{\mQ} \in \R_{> 0}$ such that
        \[
        \mQ \coloneqq \left \lbrace q_0 \in L^1(\R): 0 \leq q_0 \leq q_{\max} \text{ a.e.},~ |q_0|_{\mathrm{TV}(\R)} \leq \mC_{\mQ} \right \rbrace,
        \]
        $J:C\big([0, T]; L^1_{\text{loc}}(\R) \big) \times C\big([0, T]\times \R\big) \times L^1_{\text{loc}}(\R)\times C([0, q_{\max}], \R_{\geq 0}) \times L^1(\mathbb{R}) \rightarrow \R$ is continuous, and  $J$ \emph{does not} depend on the first component (i.e., $q$). 
        \label{item:2}
        \item \textbf{Strong convergence of the sequence of minimizers with respect to the nonlocal reach \(\eta\in\R_{>0}\) for unbounded \(\Omega\):}\\ There is a $\mC_{\mQ} \in \R_{>0}$ such that the admissible set $\mQ$ for the initial datum is \emph{a compact subset} of 
        \[
        \left \lbrace q_0 \in L^1(\R): 0 \leq q_0 \leq q_{\max} \text{ a.e.},~ |q_0|_{\mathrm{TV}(\R)} \leq \mC_{\mQ} \right \rbrace
        \]
and we define
\[
\sup_{q_{0}\in\mQ}\|q_{0}\|_{L^{1}(\R)}\eqqcolon\mC_{\mQ_{1,0}}\in\R_{\geq0}.
\]

       Moreover, we consider
        \[
        \mV \coloneqq \Big \lbrace V \in W^{2, \infty}\big([0, q_{\max}]; \R_{\geq 0}\big): V'\leqq 0,\ \lVert V \rVert_{W^{2, \infty}([0, q_{\max}])} \leq \mC_{\mV} \Big \rbrace
        \]
                rather than the $\mathcal{V}$ at the beginning of~\cref{ass:optimal_control} and postulate that
        $J:C\big([0, T]; L^1(\R)\big) \times C\big([0,T]; W^{1, 1}(\R)\big) \times L^1(\R)\times C([0, q_{\max}], \R_{\geq 0}) \times L^1(\mathbb{R}) \rightarrow \R$ is continuous in the induced topology.
        \label{item:3}
    \end{enumerate} 
\end{enumerate}
\end{assumption}

We now add further clarification and interpretation for these assumptions.
\begin{remark}[properties of admissible sets]\label{Rem:sol}
    \begin{itemize}[leftmargin=15pt]
\item The set \(\mQ\) ensures that the dynamics obey a maximum principle, which is the reason for imposing an \(L^{\infty}\) bound in all three cases. Such a bound will then be satisfied by the PDE for all times, which will become clear when presenting \cref{thm:eub}. 
\item The sets \(\mV\) and \(\Gamma\) are chosen so that they automatically give compactness in the space of continuous functions for the velocities and in \(L^{1}(\R)\) for the considered kernel class. Thereby, the growth condition in \(\Gamma\) ``close'' to infinity is essential, as the \(TV\) bound on \(\Gamma\), which is due to the monotonicity assumption, only yields compactness in \(L^{1}_{\text{loc}}\). 

\item The bound on the \(W^{1,\infty}\) norm of the admissible velocities might seem strong, but considering that it is a control in the ``coefficients'' of the PDE and that we need to have such a bound also in the limit to obtain the well-posedness of the initial value problem, it is a somewhat limited requirement.

\item The monotonicity of the kernel is another requirement for having the conservation law be well-posed on any finite time horizon, and the sign on the derivative of the velocity function is crucial, as it is also essential for the existence and uniqueness results with a maximum principle in \cref{thm:eub}. 

\item 
Concerning the additional requirements in \crefrange{item:1}{item:3} and particularly the assumptions on the objective function, we remark that in \cref{item:1}, the optimization with respect to the initial datum \(q_{0}\) is via additional \(L^{2}\) bounds, which are guaranteed by the coercitivity of the objective function in \cref{eq:coer}. The choice of regularity in the objective with respect to the second component, the nonlocal term, might be surprising, but as is well-known (and proven later), the weak convergence of the solution in \(L^{2}\) gives weak convergence in \(H^{1}\) of the nonlocal operator so that the objective function is kept as general as possible (and could even involve spatial derivatives of the nonlocal operator).

\item
In \cref{item:2}, we impose on the admissible set a \(TV\) bound regarding the initial datum. This choice is crucial in the singular limit problem considered later in \cref{sec:singular_limit}, but it might be weakened in future studies thanks to the recent \cite{coclite2025singular} (compare also \cref{sec:open_problems}). However, because the initial datum converges only on compact sets in \(L^{1}(\R)\), we have no convergence---not even locally---of the nonlocal solution, which is the reason for making the objective function independent of the solution. Due to the uniform \(L^{\infty}\) bounds, we do, however, obtain weak\(*\) convergence of \(q_{0}\) in \(L^{\infty}(\R)\), which is enough to conclude that the nonlocal operator is measured in the uniform topology.
\item Eventually, \cref{item:3} fixes the shortcoming of \cref{item:2} that we might not obtain strong convergence in \(C([0,T];L^{1}(\R))\) (i.e., on the entirety of \(\R\)) of the minimizing sequence by proposing that the initial datum converges strongly in \(L^{1}(\R)\) (which is due to the assumed compactness). Then, the objective function can indeed depend on the solution and, once more, as the nonlocal term is one regularity class higher than the solution, the objective can depend on \(W\) in \(C([0,T];W^{1,1}(\R))\). However, we then require higher regularity on the involved admissible set for the velocity, as uniform convergence will need to hold in \(W^{1,\infty}([0,q_{\max}])\), which is why we chose bounds in \(W^{2,\infty}([0,q_{\max}])\) (for this, compare particularly the \(TV\) estimate in \cref{thm:stab}).

\item Eventually, for the objective function, we currently only  assume lower weak semicontinuity in the initial datum and the solution and continuity in \(V\). This is, for instance, satisfied for classical tracking-type functionals (to name a few) like
    \begin{align}
        (q, W, q_0, \gamma, V) & \mapsto \|q(T,\cdot) - q_{\text{T}}\|_{L^{2}(\R)}^{2} \notag \\
        (q, W, q_0, \gamma, V)  & \mapsto \|q - q_{\text{d}}\|_{L^{2}((0, T) \times \R)}^{2}  \notag 
    \intertext{in the case of~\cref{ass:optimal_control}, \cref{item:1},}
        (q, W, q_0, \gamma, V) & \mapsto \big\| W(T, \cdot)  - q_\text{T}\big\|_{L^{2}(I)}^{2}\notag \\
        (q, W, q_0, \gamma, V) & \mapsto \big\| W  - q_\text{d}\big\|_{L^{2}(I_{T}\times I)}^{2} \notag 
   \intertext{for bounded $I \subseteq \R$ in the case of~\cref{ass:optimal_control}, \cref{item:2},}
        (q, W, q_0, \gamma, V) & \mapsto \lVert W(T, \cdot) - q_T \rVert_{L^1(\R)} +  \lVert q(T, \cdot) - q_d \rVert_{L^1(\R)} \notag 
  \intertext{in the case of~\cref{ass:optimal_control}, \cref{item:3}, or}
        (q, W, q_0, \gamma, V) & \mapsto \|q_{0}-q_{0,\text{d}}\|_{L^{1}(I)}^{2} \label{eq:J3}\\
        (q, W, q_0, \gamma, V)&\mapsto \|V-V_{\text{d}}\|_{L^{1}((0,q_{\max}))}^{2} \notag 
    \end{align}
        in any case. Conical combinations of these functionals for
        given \(q_{T}\in L^{2}(I),\ q_{\text{d}}\in L^{2}(I_{T}\times I),\ q_{0,d}\in L^{2}(I)\) and \(V_{\text{d}}\in L^{2}((0,q_{\max}))\) being given and 
        \(\emptyset\neq I_{T}\subset (0,T)\) being open and \(\emptyset\neq I \subseteq \R\) open, not necessarily bounded, are interesting too, as long as~\cref{ass:optimal_control} is respected. We will use some in \cref{sec:examples_numerics}.
\end{itemize}
\end{remark}
So far, it is unclear why, given \cref{ass:optimal_control}, we have a solution to the nonlocal conservation law and why such a solution does not exceed the claimed and required \(L^{\infty}\) bounds that show up in the admissible set. However, this is assured by the following theorem.
\begin{theorem}[existence, uniqueness, and \(L^{\infty}\) bounds of weak solutions]\label{thm:eub}
Let \((q_{0},V,\gamma)\in\mC\times\mV\times\Gamma\), as in any of the cases in \cref{ass:optimal_control}, be given. Then, for any \(T\in\R_{>0}\), there exists a unique weak solution of the corresponding initial value problem in \cref{eq:con},
\[
q\in C\big([0,T];L^{1}_{\textnormal{loc}}(\R)\cap L^{2}(\R)\big)\cap L^{\infty}((0,T);L^{\infty}(\R)),
\]
and \(q\) satisfies the following min/max principle:
\[
0\leq \essinf_{y\in\R}q_{0}(y)\leq q(t,x)\leq \|q_{0}\|_{L^{\infty}(\R)}\leq q_{\max},\qquad (t,x)\in (0,T)\times\R \text{ a.e.}
\]
More specifically, $q$ is given by the formula
\[
q(t, x) = q_0\big(\xi_{W}(t, x; 0)\big)\partial_2 \xi_{W}(t, x; 0),\qquad (t,x)\in (0,T)\times\R \text{ a.e.},
\]
with \(\xi_{W},\) the characteristics, being the unique solution of the following initial value problem for \((t,x,s)\in [0, T] \times \R \times [0, T]\):
\begin{equation}
\begin{aligned}
    \partial_3 \xi_{W}(t,x;s)&=V\big(W(s,\xi_W(t, x; s))\big) \\
    \xi_{W}(t, x; t) &= x.
\end{aligned}
\label{eq:characteristics}
\end{equation}
Here, \(W\in W^{1,\infty}((0,T)\times\R)\) is the unique fixed point of
\begin{equation}
W(t,x)=\int_{\xi_{W}(t,x;0)}^{\infty}\gamma\big(\xi_{W}(0,z;t)-x\big)q_{0}(z)\dd z,\quad (t, x) \in(0, T) \times \R.\label{eq:fixed_point_W}
\end{equation}
If \(q_{0}\in TV(\R)\) in addition, we also obtain
\[
q\in L^{\infty}((0,T);TV(\R)).
\]
\end{theorem}
\begin{proof}
    The existence result is given in~\cite{keimernonlocalbalance2017}. The proof for the min/max principle can be found in~\cite[Corollary~2.1]{CocliteBV2022} or~\cite[Theorem~3.1]{Keimer2023}.
\end{proof}

Next, we collect some useful bounds for the nonlocal operator \(W[q,\gamma]\). These will be crucial in the later proof of \cref{Lem:ccar}.
\begin{lemma}[uniform bounds on $W$]\label{lem:ubw}
    For some $C_{\Gamma} \in \R_{>0}$, let 
    \begin{align*}
    \Gamma_\textnormal{r} & \coloneqq \Big \lbrace \gamma \in L^1\big(\R_{\geq 0}; \R_{\geq 0} \big): \gamma\text{ is non-increasing a.e., }\lVert \gamma \rVert_{L^\infty(\R_{\geq 0})} \leq C_{\Gamma},\\
    &\qquad\qquad\qquad  \|\gamma\|_{L^{1}(\R_{\geq0})}=1\Big \rbrace\supseteq \Gamma
    \end{align*}
    be a ``relaxed'' kernel set, $(q_0, V, \gamma) \in \mQ \times \mV \times \Gamma_\textnormal{r}$ and 
    \begin{equation}\label{eq:w}
    W[q](t, x) \coloneqq \int^\infty_x \gamma(y-x)q(t, y)\dd y, \quad \forall (t, x)\in (0, T) \times \R.
    \end{equation}
    Then, it holds that
       \begin{align*}\|W[q]\|_{L^{\infty}((0,T)\times\R)}  &\leq \|q_{0}\|_{L^{\infty}(\R)}\\
        \|\partial_2 W[q]\|_{L^{\infty}((0,T)\times\R)} &\leq  \|q_{0}\|_{L^{\infty}(\R)}|\gamma|_{TV(\R_{\geq0})}\\
        \|\partial_{1}W[q]\|_{L^{\infty}((0,T)\times\R)}\|&  \leq 2 |\gamma|_{TV(\R_{>0})} \| V \|_{L^{\infty}([0, q_{\max}])} q_{\max}.
        \end{align*}
\end{lemma}

Note that these bounds do not depend on $w$ but only on the specifications of the sets $\mQ$, $\Gamma$, and $\mV$.
\begin{proof}
The first estimate is a direct consequence of \cref{thm:eub}, the second can be found in \cite[Lemma~3.1]{CocliteBV2022}, and the third can be obtained by smoothly approximating $q_0$, then recognizing that \(q(t, \cdot)\) has the same regularity for $t>0$ and differentiating \(W\) with respect to \( t\). An integration by parts quickly leads to the result.        
\end{proof}

 In the following, we will study how the characteristics as introduced in \cref{eq:characteristics} behave when the nonlocal operator varies. Similar results concerning the stability of the nonlocal operator were also obtained in \cite[Lemma~3.1]{Coclite2023} and include the self-mapping property required even to prove the existence of a fixed point as in \cref{eq:fixed_point_W}.
 
\begin{theorem}[convergence of characteristics]\label{Lem:ccar}
    Let~\cref{ass:optimal_control}, \cref{item:1} hold and $\big(q_{0, k}, V_k, \gamma_{k} \big)_{k \in \N} \subseteq \mQ \times \mV \times\Gamma$ be any sequence. Then, there exists a subsequence (that we denote by the same index), a triple $\big(q^\ast_0, V^\ast, \gamma^\ast \big) \in \mQ \times \mV \times \Gamma$, and some $\hat{W} \in W^{1, \infty}\big([0, T] \times \R\big)$ such that
    \begin{enumerate}[leftmargin=15pt]
        \item $q_{0, k} \xrightharpoonup[*]{k \rightarrow \infty} q_0^\ast \text{ in }L^\infty(\R)$,
        \item $V_k \overset{k \rightarrow \infty}{\longrightarrow} V^\ast \text{ in }C\big([0, q_{\max}]\big)$,
        \item $\gamma_{k} \overset{k \rightarrow \infty}{\longrightarrow} \gamma^\ast\text{ in }L^1\big(\R_{\geq 0}\big)$, \label{item:a}
        \item $W\left[q\big[q_{0, k}, V_k, \gamma_{k}\big] \right]\overset{k \rightarrow \infty}{ \longrightarrow} \hat{W}$ in $C\big([0, T] \times \R\big)$, \label{item:b}
        \item \label{item:c} $\xi_{W\left[q\left[q_{0, k}, V_k, \gamma_{k}\right] \right]} \overset{k \rightarrow \infty}{ \longrightarrow} \hat{\xi}$ in $C\big([0, T] \times \R \times [0, T]\big)$, where $\hat{\xi}(t, x; \cdot)$ solves, for \((t,x,s)\in [0,T]\times\R\times[0,T]\),
        \begin{align*}
        \partial_3 \hat{\xi}(t, x; s) &= V^\ast\big(\hat{W}(\hat{\xi}(t, x; s)) \big)\\
        \hat{\xi}(t, x; t) &= x,
        \end{align*}

        \item $\hat{W} = W\left[q\big[q_{0, k}, V_k, \gamma_{k}\big] \right]$.\label{item:d}
    \end{enumerate}
\end{theorem}

\begin{proof}
        For the sake of brevity, we define
        \begin{equation}\label{eq:abb}
        W_k \coloneqq W\left[q\big[q_{0, k}, V_k, \gamma_{k}\big] \right], \quad W^\ast \coloneqq W\left[q\big[q_0^*, V^*, \gamma^*\big] \right],
        \end{equation}
        for all $k \in \N$.
    \begin{enumerate}[leftmargin=15pt]
        \item This is mostly a standard compactness result: One can use Banach--Alaoglu~\cite[8.5]{alt2016eng} for $\big( q_{0, k}\big)_{k \in \N}$, which is bounded in $L^\infty$, and Arzelà--Ascoli~\cite[4.12]{alt2016eng} for $V_k$. 
        Convergence of $\gamma_k$ is due to the Riesz--Kolmogorov theorem~\cite[4.16]{alt2016eng}. 
        \item The Arzelà--Ascoli theorem is applicable because of the bounds that we found in~\cref{lem:ubw}, so one obtains, on a subsequence, the existence of the named \(\hat{W}\) of corresponding regularity.
        \item To see this, observe for $(t, x, s) \in [0, T] \times \R \times [0, T]$, where we assume without loss of generality that $s \leq t$, that
        \begin{align*}
            &\big|\xi_{W_k}(t, x; s) - \hat{\xi}(t, x; s)\big| \\
            & = \left \lvert \int^t_s V_k\big(W_k(r, \xi_{W_k}(t, x; r) \big) - V^\ast\big(\hat{W}(r, \hat{\xi}(t, x; r)) \big)~\dd r \right \rvert. \\
            \intertext{Adding zeros leads to}
            & \leq \|V_k' \|_{L^\infty([0, q_{\max}])}  \left \lvert \int^t_s W_k(r, \xi_{W_k}(t, x; r)  - W_k(r, \hat{\xi}(t, x; r))\dd r \right \rvert \\
            & \quad + \|V_k' \|_{L^\infty([0, q_{\max}])} \left \lvert \int^t_s W_k(r, \hat{\xi}(t, x; r))  - \hat{W}(r, \hat{\xi}(t, x; r))~\dd r \right \rvert \\
            & \quad + \left \lvert \int^t_s V_k\big(\hat{W}(r, \hat{\xi}(t, x; r)) \big) - V^\ast\big(\hat{W}(r, \hat{\xi}(t, x; r)) \big)~\dd r \right \rvert, \\
            \intertext{and using~\cref{lem:ubw} for the emerging \(\partial_{2}W_{k}\), as well as the bounds on $V_k'$ as in \cref{ass:optimal_control}, we obtain}
            & \leq  q_{\max} C_{\mV}|\gamma_k |_{\mathrm{TV}(\mathbb{R}_{\geq 0})}  \int^t_s \lvert \xi_{W_k}(t, x; r) - \hat{\xi}(t, x; r) \rvert~\dd r\\
            &\quad + C_{\mV} T  \| W_k - \hat{W} \|_{L^\infty((0, T) \times \R)} + T \sup_{r \in [0, q_{\max}]} \| V_k - V^\ast \|_{L^\infty([0, q_{\max}])}.
        \end{align*}
        Gr\"{o}nwall's Lemma and the uniform bound $C_\Gamma$ on the total variation of $\gamma_k$, as in \cref{ass:optimal_control}, yield
        \begin{align*}
        \big\|\xi_{W_k} - \hat{\xi} \big\|_{L^{\infty}((0,T)\times\R\times(0,T))} & \leq \Big( C_{\mV} T \| W_k - \hat{W} \|_{L^\infty((0, T) \times \R)} \\
        & \qquad + T  \| V_k - V^\ast \|_{L^\infty([0, q_{\max}])} \Big) \cdot \exp\left(C_{\mV}q_{\max} C_{\Gamma}T \right).
        \end{align*}
        The desired uniform convergence of \(\xi_{W_{k}}\) to \(\xi_{\hat{W}}\) follows as \(V_{k}\) and \(W_{k}\) converge uniformly, as obtained before in \cref{item:a} and \cref{item:b}.
        \item For $\hat{W}$ to have the desired shape, it must hold that
        \begin{equation}\label{eq:hatW}
        \hat{W}(t, x) = \int^{\infty}_{\hat{\xi}(t, x; 0)} q_0^\ast(z) \gamma^\ast \big(\hat{\xi}(0, z; t)-x\big) ~\dd z
        \end{equation}
        for all $(t, x) \in (0, T) \times \R$. We prove this by showing that $W_k$ converges pointwise to the RHS of~\cref{eq:hatW}. This yields the result if we combine it with the uniform convergence from \cref{item:b}.
        \item Now, consider, for $(t, x) \in [0, T] \times \R$,
        \begin{small}
        \begin{align}
            & \left \lvert W_k(t, x) - \int^{\infty}_{\hat{\xi}(t, x; 0)}\!\! q_0^\ast(z) \gamma^\ast \big(\hat{\xi}(0, z; t)-x\big) ~\dd z \right \rvert \notag\\
            & = \bigg \lvert \int^{\infty}_{\xi_{W_k}(t, x; 0)}\!\!\! q_{0, k}(z) \gamma_{k}\big(\xi_{W_k}(0, z; t)-x\big)~\dd z - \int^{\infty}_{\hat{\xi}(t, x; 0)}\!\!\! q_0^\ast(z) \gamma^\ast\big(\hat{\xi}(0, z; t)-x\big)~\dd z \bigg \rvert.\notag
            \intertext{Add zeros:}
            & = \bigg \lvert \int^{\infty}_{\xi_{W_k}(t, x; 0)}\!\!\!\!\!\!\! q_{0, k}(z) \gamma_{k}\big(\xi_{W_k}(0, z; t)-x\big)\dd z - \int^{\infty}_{\hat{\xi}(t, x; 0)}\!\!\!\!\!\!\! q_{0, k}(z) \gamma_{ k}\big(\xi_{W_k}(0, z; t)-x\big)\dd z \bigg \rvert \label{eq:conv1}\\
            & \quad + \bigg \lvert \int^{\infty}_{\hat{\xi}(t, x; 0)}\!\!\!\!\!\! q_{0, k}(z) \gamma_{k}\big(\xi_{W_k}(0, z; t)-x\big)\dd z - \int^{\infty}_{\hat{\xi}(t, x; 0)}\!\!\!\!\!\! q_{0, k}(z) \gamma_{ k}\big(\hat{\xi}(0, z; t)-x\big)\dd z \bigg \rvert \label{eq:conv2}\\
            & \quad + \bigg \lvert \int^{\infty}_{\hat{\xi}(t, x; 0)}\!\!\!\!\!\! q_{0, k}(z) \gamma_{k}\big(\hat{\xi}(0, z; t)-x\big)\dd z -\int^{\infty}_{\hat{\xi}(t, x; 0)}\!\!\!\!\!\! q_{0, k}(z) \gamma^\ast \big(\hat{\xi}(0, z; t)-x\big)\dd z \bigg \rvert\label{eq:conv3} \\
            & \quad + \left \lvert \int^{\infty}_{\hat{\xi}(t, x; 0)}\!\!\!\!\!\! q_{0, k}(z) \gamma^\ast\big(\hat{\xi}(0, z; t)-x\big)\dd z - \int^{\infty}_{\hat{\xi}(t, x; 0)}\!\!\!\!\!\! q_0^\ast(z) \gamma^\ast \big(\hat{\xi}(0, z; t)-x\big)\dd z \right \rvert. \label{eq:conv4}
            \intertext{Now, in order to simplify~\cref{eq:conv1}, we use the substitution formula~\cite[Lemma~2.4]{Coron2020} in~\cref{eq:conv2}, and in~\cref{eq:conv3}, estimate $\lvert q_{0, k} \rvert \leq q_{\max}$ and substitute $y \coloneqq \hat{\xi}(0, z; t)-x$, yielding}
            & \leq \left \lvert \int_{\min \lbrace \xi_{W_k}(t, x; 0), \hat{\xi}(t, x; 0) \rbrace}^{\max \lbrace \xi_{W_k}(t, x; 0), \hat{\xi}(t, x; 0) \rbrace} q_{0, k}(z) \gamma_{k}\big(\xi_{W_k}(0, z; t) - x \big)~\dd z \right \rvert \label{eq:mes} \\
            & \quad + q_{\max} C_{\Gamma} \big \lVert \xi_{W_k} - \hat{\xi} \big \rVert_{C([0, T] \times \R \times [0, T])} \notag\\
            & \quad + q_{\max} \exp\left( TC_{\mV} q_{\max} C_{\Gamma}\right) \left \lVert  \gamma_k - \gamma^\ast \right \rVert_{L^1(\R_{\geq 0})} \notag \\
            & \quad + \left \lvert \int^{\infty}_{\hat{\xi}(t, x; 0)} \big(q_{0, k}(z) -  q_0^\ast(z)\big)\gamma^\ast\big(\hat{\xi}(0, z; t)-x\big)  ~\dd z \right \rvert, \notag 
            \intertext{and estimating \cref{eq:mes},}
            & \leq 2q_{\max} C_{\Gamma} \big\lVert \xi_{W_k} - \hat{\xi}\big\rVert_{C([0, T] \times \R \times [0, T])} \label{eq:fc1}\\
            & \quad + q_{\max} \exp\left( TC_{\mV} q_{\max} C_{\Gamma}\right) \left \lVert  \gamma_k - \gamma^\ast \right \rVert_{L^1(\R_{\geq 0})} \label{eq:fc3} \\
            & \quad + \left \lvert \int^{\infty}_{\hat{\xi}(t, x; 0)} \big(q_{0, k}(z) -  q_0^\ast(z)\big)\gamma^\ast\big(\hat{\xi}(0, z; t)-x\big)  ~\dd z \right \rvert. \label{eq:fc4}
    \end{align}
    \end{small}
    \item Now,~\cref{eq:fc1} converges to $0$ because of the uniform convergence of the characteristics,~\cref{eq:fc3} converges because of the $L^1$-convergence of $\gamma_k$, and~\cref{eq:fc4} converges because of the weak$\ast$ convergence of $q_{0, k}$. In total, 
    \[
    \hat{W}(t, x) = \lim_{k \rightarrow \infty} W_k(t, x) = \int^{\infty}_{\hat{\xi}(t, x; 0)} q_0^\ast(z) \gamma^\ast \big(\hat{\xi}(0, z; t)-x\big) ~\dd z
    \]
    for all $(t, x) \in [0, T] \times \R$ (and even uniformly because of \cref{item:b,item:c}). This proves \cref{item:d}. 
     \end{enumerate}
\end{proof}

 When later dealing with the existence of optimizers in the case of \cref{ass:optimal_control}, \cref{item:3}, we require another stability estimate on nonlocal solutions when changing the initial datum, the velocity, and the nonlocal kernel in a suitable topology. This is provided by the following theorem.
\begin{theorem}[strong stability of the nonlocal conservation law]\label{thm:stab}
    Let~\cref{ass:optimal_control}, \cref{item:3} hold and let $(q_0, V, \gamma), (\hat{q}_0, \hat{V}, \hat{\gamma}) \in \mathcal{Q} \times \mathcal{V} \times \Gamma$ be arbitrary. Then, any weak solution with a datum in this set satisfies the following uniform \(TV\) bound for \(t\in[0,T]\):
\begin{equation}   
\begin{split}&\sup_{\substack{(q_{0},V,\gamma)\in\\\mQ\times\mV\times\Gamma}}|q[q_{0},V,\gamma](t,\cdot)|_{TV(\R)}\leq \mC_{\textnormal{TV}}(t)\\
    &\mC_{\textnormal{TV}}(t) \coloneqq \big(\mC_{\mQ}+t\|V''\|_{L^{\infty}([0,q_{\max}])}\mC_{\mQ_{1,0}}q_{\max}^{2}\mC_{\Gamma}^{2}\Big)\mathrm{e}^{2t\|V'\|_{L^{\infty}([0,q_{\max}])}q_{\max}\mC_{\Gamma}}.
\end{split}
\label{eq:TV_bounds}
\end{equation}
    Additionally, the following stability result holds for any \(t\in[0,T]\):
    \begin{small}
    \begin{equation}
    \begin{split}
    &\| q[q_0, V, \gamma](t, \cdot) - q[\hat{q}_0, \hat{V}, \hat{\gamma}](t, \cdot) \|_{L^1(\R)} \\
    &\leq \bigg(\|q_{0}-\tilde{q}_{0}\|_{L^{1}(\R)}+\Big(tq_{\max}\mC_{\mQ_{1,0}}+\int_{0}^{t}\mC_{\textnormal{TV}}(s)\dd s\Big)\|V-\hat{V}\|_{W^{1,\infty}([0,q_{\max}])}\\
    &\ +\Big(2\!\!\int_{0}^{t}\!\!\mC_{\textnormal{TV}}(s)\dd s\mV q_{\max}+t\mV\mC_{\mQ_{1,0}}\mC_{\Gamma}\Big)\|\gamma-\hat{\gamma}\|_{L^{1}(\R_{\geq0})}\bigg)\mathrm{e}^{\mV\mC_{\Gamma}t(\mC_{\textnormal{TV}}(T)+q_{\max}(1+q_{\max}\mC_{\Gamma}))} .
    \end{split} \notag 
    \end{equation}
    \end{small}
    
\end{theorem}

\begin{proof}
We only sketch the proof and rely on the fact that solutions can be approximated smoothly; that is, we can assume that we have strong solutions.
Then, the total variation bound is straightforward and can be made uniform by using the uniform bounds in \cref{ass:optimal_control},\ \cref{item:3}.

Concerning the stability, we look into the time derivative of the \(L^{1}\)-norm of the difference. For \(t\in[0,T]\),
with the abbreviation \[q(t,\cdot)\coloneqq q[q_0, V, \gamma](t, \cdot),\qquad \hat{q}(t,\cdot)\coloneqq q[\hat{q}_0, \hat{V}, \hat{\gamma}](t, \cdot) \text{ on } \R\]
and with \(\hat{W}\) denoting the nonlocal operator with kernel \(\hat{\gamma}\) and \(W\) the nonlocal operator with kernel \(\gamma\), we obtain
\begin{align*}
    &\partial_{t}\|q(t,\cdot)-\hat{q}(t,\cdot)\|_{L^{1}(\R)}\\
    &=-\int_{\R}\sgn\big(q-\hat{q}\big)\partial_{x}\big(V(W[q])q-\hat{V}(\hat{W}[\hat{q}])\hat{q}\big)\dd x.
    \intertext{Adding zeros leads to}
    &=-\int_{\R}\sgn\big(q-\hat{q}\big)\partial_{x}\big(V(W[q])q-\hat{V}(W[q])q\big)\dd x\\
    &\quad -\int_{\R}\sgn\big(q-\hat{q}\big)\partial_{x}\big(\hat{V}(W[q])q-\hat{V}(\hat{W}[q])q\big)\dd x\\
    &\quad -\int_{\R}\sgn\big(q-\hat{q}\big)\partial_{x}\big(\hat{V}(\hat{W}[q])q-\hat{V}(\hat{W}[\hat{q}])q\big)\dd x\\
    &\quad -\int_{\R}\sgn\big(q-\hat{q}\big)\partial_{x}\big(\hat{V}(\hat{W}[\hat{q}])q-\hat{V}(\hat{W}[\hat{q}])\hat{q}\big)\dd x.
    \intertext{We next estimate the RHS terms by taking advantage of the definition of the nonlocal operator \cref{eq:nonlocal_operator}, the previously derived bounds in \cref{lem:ubw}, and \(\partial_{2}W[q]\equiv W[\partial_{2}q]\):}
    &\leq \|V'-\hat{V}'\|_{L^{\infty}([0,q_{\max}])}\|\partial_{2}W[q]\|_{L^{\infty}((0,T)\times\R)}\|q(t,\cdot)\|_{L^{1}(\R)}\\
    &\quad + |q(t,\cdot)|_{TV(\R)}|\|V-\hat{V}\|_{L^{\infty}([0,q_{\max}])}\\
    &\quad +|q(t,\cdot)|_{TV(\R)}\|V'\|_{L^{\infty}([0,q_{\max}])}\|q(t,\cdot)\|_{L^{\infty}(\R)}\|\gamma-\hat{\gamma}\|_{L^{1}(\R_{\geq0})}\\
    &\quad +\|\hat{V}'\|_{L^{\infty}([0,q_{\max}])}\|q(t,\cdot)\|_{L^{\infty}(\R)}|q(t,\cdot)|_{TV(\R)}\|\gamma-\tilde{\gamma}\|_{L^{1}(\R_{\geq0})}\\
    &\quad +\|\hat{V}''\|_{L^{\infty}([0,q_{\max}])}\|\partial_{2}\hat{W}[q]\|_{L^{1}(\R)}\|\gamma-\hat{\gamma}\|_{L^{1}(\R_{\geq0})}\\
    &\quad+ \|\hat{V}'\|_{L^{\infty}([0,q_{\max}])}|q(t,\cdot)|_{TV(\R)}\|\hat{\gamma}\|_{L^{\infty}(\R_{\geq0})}\|q(t,\cdot)-\hat{q}(t,\cdot)\|_{L^{1}(\R)}\\
    &\quad +\|\hat{V}'\|_{L^{\infty}([0,q_{\max}])}\|q(t,\cdot)\|_{L^{\infty}(\R)}|\hat{\gamma}|_{TV(\R_{\geq0})}\|q(t,\cdot)-\hat{q}(t,\cdot)\|_{L^{1}(\R)}\\
    &\quad +\|\hat{V}''\|_{L^{\infty}([0,q_{\max}])}\|q(t,\cdot)\|_{L^{\infty}(\R)}\|\partial_{2}\hat{W}[q]\|_{L^{\infty}((0,T)\times \R)} \\
    & \qquad \cdot \|\hat{\gamma}\|_{L^{1}(\R)}\|q(t,\cdot)-\tilde{q}(t,\cdot)\|_{L^{1}(\R)}.
    \intertext{Next, we make the estimate uniform by applying the bounds guaranteed in \cref{ass:optimal_control}, \cref{item:3} and the uniform \(TV\) bounds in \cref{eq:TV_bounds}:}
    &\leq q_{\max}\mC_{\mQ_{1,0}}\|V'-\hat{V}'\|_{L^{\infty}([0,q_{\max}])}+ \mC_{\text{TV}}(t)\|V-\hat{V}\|_{L^{\infty}([0,q_{\max}])}\\
    &\quad +2\mC_{\text{TV}}(t)\mV q_{\max}\|\gamma-\hat{\gamma}\|_{L^{1}(\R_{\geq0})}+\mV\mC_{\mQ_{1,0}}\mC_{\Gamma}\|\gamma-\hat{\gamma}\|_{L^{1}(\R_{\geq0})}\\
    &\quad+ \mV\mC_{\text{TV}}(t)\mC_{\Gamma}\|q(t,\cdot)-\hat{q}(t,\cdot)\|_{L^{1}(\R)}
    +\mV q_{\max}\mC_{\Gamma}\big(1+q_{\max}\mC_{\Gamma}\big)\|q(t,\cdot)-\hat{q}(t,\cdot)\|_{L^{1}(\R)}
\end{align*}
from which, together with Gr\"{o}nwall's inequality, the conclusion follows.
\end{proof}

The next theorem is one of the main contributions, the existence of a minimizer under the different assumptions mentioned in \cref{ass:optimal_control}.
\begin{theorem}[existence of a minimizer of \cref{defi:optimal_control_problem}]\label{Thm:Ex}
    There exists at least one solution $(q_0^\ast, V^\ast, \Gamma^\ast) \in \mQ \times \mV \times \Gamma$ to the problem~\eqref{Eq:O} in \cref{defi:optimal_control_problem} given any of the cases in \cref{ass:optimal_control}.
\end{theorem}
\begin{proof}
    Consider a minimizing sequence $\big(q_{0, k}, V_k, \gamma_k \big)_{k \in \N} \subseteq \mQ \times \mV \times\Gamma$.
    We set
    \[
    q_k \eqqcolon q[q_{0, k}, V_k, \gamma_k], \quad W_k \eqqcolon W[q_k], \quad q^* \eqqcolon q[q_{0}^*, V^*, \gamma^*], \quad W ^* \eqqcolon W[q^*].
    \]
    We then distinguish the three different cases:
    \begin{description}
        \item[\cref{ass:optimal_control}, \cref{item:1}:] Thanks to the coercivity condition~\cref{eq:coer} on $J$, the sequence $(q_{0,k})_{k \in \mathbb{N}}$  is uniformly bounded $L^2(\R)$, exploiting the solution from formula~\cref{thm:eub}.
    Then, like in~\cref{Lem:ccar}, there exists a subsequence and some $(q_0^\ast, V^\ast, \Gamma^\ast) \in \mQ \times \mV \times \Gamma$ such that
    \begin{itemize}
        \item $q_{0, k} \xrightharpoonup{k\rightarrow\infty} q_0^\ast \text{ in }L^2(\R)$ and $q_{0, k} \xrightharpoonup[\star]{k\rightarrow\infty} q_0^\ast$ in $L^\infty(\R)$,
        \item $V_k \overset{k \rightarrow \infty}{\longrightarrow} V^\ast \text{ in }C\big([0, q_{\max}] \big)$, and
        \item $\gamma_{k} \overset{k \rightarrow \infty}{\longrightarrow} \gamma^\ast\text{ in }L^1\big(\R_{\geq 0}\big)$.
    \end{itemize}
    Let us also use the abbreviations in~\cref{eq:abb} again.
    We want to prove that $q_{k}(t, \cdot)\overset{k \rightarrow \infty}{\rightharpoonup}q^\ast(t, \cdot)$ in $L^2(\R)$ and $W_k(t, \cdot) \rightarrow W^*(t, \cdot)$ in $H^1(\R)$ for all $t>0$, which will be sufficient for existence because of the weak lower semicontinuity of the objective functional. Here, $q^\ast$ denotes the solution to the nonlocal conservation law with velocity $V^\ast$, initial datum $q_0^\ast$, and kernel $\gamma^\ast$. 
    Now, we prove that the solution to \(q_{k}\) with initial datum \(q_{0,k},\) velocity \(V_{k}\), and nonlocal kernel \(\gamma_{k}\) converges weakly in \(L^{2}\) to the solution \(q^{*}\) with initial datum \(q_{0}^{*},\) velocity \(V^{*}\), and nonlocal kernel \(\gamma^{*}\). To this end, choose $g \in L^2(\R)$ and \(t\in[0,T]\) fixed, and consider
    \begin{align}
        & \left \lvert \int_{\R} g(x) \big( q_{k}(t, x) - q^{*}(t, x) \big)\dd x \right \rvert \notag\\
        & = \bigg \lvert \int_{\R}g(x)\big(q_{0, k}(\xi_{W_k}(t, x; 0))\partial_2\xi_{W_k}(t, x; 0) \notag  \\ 
        &\qquad\qquad\qquad  - q_0^\ast(\xi_{\hat{W}}(t, x; 0))\partial_2\xi_{\hat{W}}(t, x; 0) \big) \dd x \bigg \rvert. \notag
        \intertext{Substitute $y = \xi_{W_k}(t, x; 0)$ and $y = \xi_{\hat{W}}(t, x; 0)$, respectively:}
        & = \left \lvert \int_{\R}g\big(\xi_{W_k}(0, y; t)\big)q_{0, k}(y) - g\big(\xi_{\hat{W}}(0, y; t)\big)q_0^\ast(y)\dd y\right \rvert.\notag
        \intertext{Add a zero:}
        & \leq \left \lvert\int_{\R}g\big(\xi_{W_k}(0, y; t)\big)q_{0, k}(y) - g\big(\xi_{\hat{W}}(0, y; t)\big)q_{0, k}(y)\dd y \right \rvert \notag\\
        & \quad +  \left \lvert \int_{\R}g\big( \xi_{\hat{W}}(0, y; t)\big)q_{0, k}(y) - g\big(\xi_{\hat{W}}(0, y; t)\big)q^\ast_{0}(y)\dd y \right \rvert \notag\\
        & \leq \big\|g\big(\xi_{W_k}(0, \cdot; t)\big)-g\big(\xi_{\hat{W}}(0, \cdot; t)\big)\big\|_{L^{2}(\R)}\|q_{0, k}\|_{L^{2}(\R)}\label{eq:g_xi}\\
        & \quad +  \left \lvert \int_{\R}g\big( \xi_{\hat{W}}(0, y; t)\big)q_{0, k}(y) - g\big(\xi_{\hat{W}}(0, y; t)\big)q^\ast_{0}(y)\dd y \right \rvert. \label{eq:wk}
    \end{align}
    Now, weak convergence in $L^2$ ensures that~\cref{eq:wk} tends to $0$ as $k \rightarrow \infty$ and \cref{eq:g_xi} converges to zero as \(k\rightarrow\infty\) because of the uniform convergence of \(\xi_{W_{k}}\), as outlined in \cref{Lem:ccar}.
    Now, we still have to show the weak convergence of $W_k$ to $W^*$ in $H^1(\R)$. It is sufficient to prove that such a weak limit exists, as we know from the aforementioned uniform convergence of $W_k$ to $W^*$ that the limit will be $W^*$. We do this by showing that $\big(W_k(t, \cdot)\big)_{k \in \mathbb{N}}$ and $\big(\partial_x W_k(t, \cdot)\big)_{k \in \mathbb{N}}$ are both uniformly bounded in $L^2$ and applying Banach--Alaoglu~\cite[8.10]{alt2016eng}.
    The argument is almost identical for both sequences, so we only prove it for $W_k$. Let $k \in \mathbb{N}$ and $t \in (0, T)$ be arbitrary. Then, we can infer by using Young's convolution inequality~\cite[4.13 (2)]{alt2016eng} that
    \begin{align*}
        \| W_k(t, \cdot) \|^2_{L^2(\R)} &  \leq \int_{\R} \left( \int_{\R}\! \left | \gamma_k(y-x)q_k(t, y) \right|\dd y \right)^{2} \! \leq  \lVert \gamma_k \rVert_{L^1(\R)}^2 \lVert q_k(t, \cdot) \rVert_{L^2(\R)}^2,
    \end{align*}
    $q_k$ is uniformly bounded in $L^2$ due to its weak convergence~\cite[8.3 Remarks, (5)]{alt2016eng}, and $\gamma_k$ is uniformly bounded in $L^1$ due to the definition of $\Gamma$. For $\partial_x W_k$, the total variation bounds from the definition of $\Gamma$ will have to be used.
    Finally, we conclude the argument by using the inequality~\cref{eq:lsem} for the minimal sequence.
    \item[\cref{ass:optimal_control}, \cref{item:2}:] We may then extract a minimal sequence (denoted by the same index) such that there is a $(q_0^*, V^*, \Gamma^*)$ satisfying
    \begin{itemize}
        \item $q_{0, k} \overset{k \rightarrow \infty}{\longrightarrow} q_0^*$ in $L^1_{\text{loc}}(\R)$ (\cite[Theorem~5.5]{evans}) and  $q_{0, k} \xrightharpoonup[*]{k \rightarrow \infty} q_0^\ast \text{ in }L^\infty(\R)$,
        \item $V_k \overset{k \rightarrow \infty}{\longrightarrow} V^*$ in $C\big([0, q_{\max}]\big)$,
        \item $\gamma_k \overset{k \rightarrow \infty}{\longrightarrow} \gamma^*$ in $L^1(\R)$.
    \end{itemize}
    Then, note the uniform convergence of $W_k$ to $W^*$ due to~\cref{Lem:ccar}, which implies convergence in $L_{\text{loc}}^1$.
    Then, continuity of the objective functional in said spaces finally yields optimality, similar to the first case.
    \item[\cref{ass:optimal_control},\ \cref{item:3}:] In this case, one can argue as in the case for \cref{item:2}, this time using the compactness of $(q_{0, k})$ in $L^1(\R)$ and the stability result~\cref{thm:stab} to infer convergence of $q_k$ to $q^*$ in $L^1(\R)$. To use the continuity of the objective function, it is still necessary to prove that $W_k \rightarrow W$ in $H^1$. However, this can be shown directly by estimating $\lVert W_k - W^* \rVert_{H^1(\R)}$ directly by again using Young's convolution inequality and adding a zero to deal with the difference in $\gamma_k$ and $\gamma^*$. 
    \end{description}
    Therefore, $\big(q_0^\ast, V^\ast, \gamma^\ast \big)\in\mQ\times\mV\times\Gamma$ is a solution to our optimization problem. 
\end{proof}

\section{Singular limit for the optimal control problem}\label{sec:singular_limit}
As mentioned in the introduction, we now look into the singular limit behavior of solutions to the optimal control problem in \cref{defi:optimal_control_problem}, that is, the behavior of optimizers when the nonlocal kernel, here the exponential kernel, tends to a Dirac distribution. The goal is to show convergence of the optimizer found in~\cref{Thm:Ex} in some sense to an optimizer of the local optimal control problem.

We will leave out an optimization with respect to the nonlocal kernel $\gamma$, as this will vanish when the kernel tends to a Dirac distribution. We will conduct this analysis for the choice of the exponential kernel~\cite{CocliteSL2022} only, as this is more compatible with the singular limit results. We will later in this section lay out ideas for generalizations. 
We first define the local optimal control problem.
\begin{definition}[local optimal control problem]\label{def:locp}
We call the optimal control problem
\begin{equation}\label{eq:opq1}
\begin{aligned}
      \min_{\substack{q_0 \in \hat{\mQ}\\ V \in \mV}} &  J(q[q_0, V], q_0, V)\\
        \text{where \(q[q_{0},V]\) denotes the local entropy}& \text{ solution of} \\
        \partial_t q(t, x) + \partial_x (V(q)q) &= 0, &&\text{ on }(0, T) \times \R,\\
        q_\eta(0, \cdot)& \equiv q_0, &&\text{ on } \R
\end{aligned}    
\end{equation}
with \(\hat{\mQ} \subseteq \mQ \) yet to be stated, \(\mQ \) and \(\mV\) as in \cref{ass:optimal_control}, and \(I\subset\R\) open and bounded, the \emph{local} optimal control problem.
\end{definition}
    
\begin{definition}[the nonlocal optimal control problem in the singular limit context]
We call the following problem the 
\emph{nonlocal optimal control problem} in the singular limit context:
\begin{equation}\label{eq:opq}
    \begin{aligned}
        \min_{\substack{q_0 \in \hat{\mQ} \\ V \in \mV}}  J\big(W_\eta[q_\eta[q_0, V]], q_0, V\big)&\\
        \text{s.t.\ the weak solution of}&\\
        \partial_t q_\eta + \partial_x \big(V\left(W_\eta[q_{\eta}[q_0, V]](t, x)\right)q_\eta \big) &= 0, &&\text{ in }(0, T) \times \R, \\
        q_\eta(0, \cdot) &\equiv q_0, &&\text{ on }x\in\R,
    \end{aligned}
\end{equation}
for $J:L^1_{\textnormal{loc}}((0, T) \times \mathbb{R}) \times L^1_\textnormal{loc}(\mathbb{R}) \times C\big([0, q_{\max}]\big) \rightarrow \mathbb{R}$ with
\[
W_{\eta}[q_{\eta}[q_0, V]](t,x)\coloneqq \tfrac{1}{\eta}\int_{x}^{\infty} \gamma\big(\tfrac{y-x}{\eta}\big) q(t, y)\dd y \quad \forall (t, x) \in (0, T) \times \R
\]
being the nonlocal operator. Here, $\hat{\mathcal{Q}}$ is as in~\cref{def:locp}, and \(\mV\) is as in~\cref{ass:optimal_control}. 
\end{definition}

With the problem stated, we want to justify the assumptions that will be made in~\cref{theo:singular_limit}.
\begin{remark}[assumptions in singular limit process]
    In the ensuing~\cref{theo:singular_limit}, we will have to operate under~\cref{ass:optimal_control},~\cref{item:2}, as the crucial singular limit error estimate~\cite[Theorem~1.3]{Colombo2023} used in the theorem is about convergences of $W_\eta$ and not $q_\eta$. Moreover, it requires total variation bounds on the initial datum. If they are uniform, they lead to compactness of a sequence of initial data (which is crucial for the singular limit result below) in $L_{\text{loc}}^1(\R)$ but not in $L^1(\R)$. Therefore, to avoid excessive assumptions on the initial datum, we will only be considering tracking functionals on bounded intervals. This essentially amounts to~\cref{ass:optimal_control},~\cref{item:2}.
\end{remark}

\begin{theorem}[singular limit]\label{theo:singular_limit}
    Let~\cref{ass:optimal_control},~\cref{item:2} hold and let $\gamma \in \Gamma$ ($\Gamma$ as in~\cref{ass:optimal_control}) be convex such that $x\mapsto x\gamma(x)$ is in $L^1(\mathbb{R})$.
    For $\eta \in \R_{>0}$, let $\big( q_{0, \eta}^\ast, V^\ast_{\eta} \big)$ denote an optimal solution to the problem~\cref{eq:opq}. Then, any accumulation point $\big( q_0^\ast, V^\ast \big)$ of the sequence $\big( q_{0, \eta}^\ast, V^\ast_{\eta} \big)_{\eta > 0} \subseteq \hat{\mQ} \times \hat{V}$ with respect to the strong $L^1_{\textnormal{loc}}(\R) \times C(\mathbb{R})$ topology is a solution (i.e., optimizer) for the local optimal control problem in~\cref{def:locp}.
\end{theorem}
\begin{proof}
    Let $\big( q_0^\ast, V^\ast\big)$ be an accumulation point of $\big(q_{0, \eta}^\ast, V_\eta^\ast  \big)_{\eta > 0}$ in the described sense. We again denote the associated subsequence by $\big(q_{0, \eta}^\ast, V_\eta^\ast \big)_{\eta > 0}$. Moreover, we need to denote the explicit dependencies of $w_\eta$, $q_\eta$, and $q$:
    \[
    q_\eta = q_\eta[q_0, V],\quad W_\eta =  W_\eta [q_\eta[q_0, V]], \quad q  = q[q_0, V]. 
    \]
    Here, for any $(q_0, V) \in \hat{\mQ} \times \mV$, $q[q_0, V]$ denotes the entropy solution to the local conservation law and $q_\eta[q_0, V]$ that to the nonlocal conservation law.
    First, we want to prove that 
    \[
    \lim_{\eta \rightarrow 0}\|  W_\eta [q_\eta[q_{0, \eta}^*, V_\eta^*]]  - q[q_0^*, V^* ] \|_{L^1(I)} = 0
    \]
    for any bounded interval $I \subseteq \mathbb{R}$. This is being done by adding a zero ($\eta > 0$, $t \in (0, T)$, and $I \subseteq \mathbb{R}$ is a bounded interval):
    \begin{align}
        & \left \| W_\eta [q_\eta[q_{0, \eta}^*, V_\eta^*]](t, \cdot) - q[q_0^*, V^*](t, \cdot) \right\|_{L^1(I)} \notag  \\
        & \leq \| W_\eta[q_\eta[q_{0, \eta}^*, V_\eta^*]](t, \cdot) - q[q_{0, \eta}^*, V_\eta^*](t, \cdot) \|_{L^1(I)} \label{eq:slo1}\\
        & \quad + \| q[q_{0, \eta}^*, V_\eta^*](t, \cdot) - q[q_{0}^*, V^*](t, \cdot) \|_{L^1(I)}. \label{eq:slo2}\\
        \intertext{We apply the singular limit error estimate~\cite[Theorem~1.3]{Colombo2023} and the local stability estimate~\cite[Theorem~3.1]{Bouchut1998} (for an appropriate constant $C>0$ and $x_m \coloneqq \tfrac{1}{2}\left(\sup I + \inf I \right)$, $R \coloneqq \tfrac{1}{2}\mathrm{diam}(I)$):}
        & \leq C\big( \eta + \sqrt{\eta T}\big)\mC_{\hat{\mQ}} + \| q_{0, \eta}^* - q_0^* \|_{L^1([x_m - R-\mC_{\mV}T, x_m + R + \mC_{\mV}T])} \notag \\
        & \quad + C \big( 2(R + \mC_{\mV}T) \mC_{\hat{\mQ}} T \| V_\eta^* - V^* \|_{C([0, q_{\max}])} \big)^{\frac{1}{2}}.\notag  \\
        \intertext{Note that none of the constants depend on $\eta$, and $(q_0^*, V^*)$ is an accumulation point, whence we can conclude convergence as $\eta \rightarrow 0$:}
        & \overset{\eta \rightarrow 0}{\longrightarrow} 0 \notag .
    \end{align}
    Finally, let $(q_0, V) \in \hat{\mQ} \times \mV$ be arbitrary. Then, we can infer the following from the optimality of $(q_{0, \eta}^*, V_\eta^*)$ for every $\eta$, the singular limit~\cite[Theorem~1.3]{Colombo2023} (to be applied on the RHS), the continuity of $J$ (in $L^2_\text{loc}$ but thus also in $L^1_\text{loc}$ owing to Hölder's inequality), and the above limit:
    \begin{align*}
    & J\big(W_\eta[q_{0, \eta}^*, V_\eta^*], q_{0, \eta}^*, V_\eta^* \big)  \leq J\big(W_\eta[q_0, V_0],q_0, V_0\big)   \\
    \overset{\eta \rightarrow 0}{\implies} & J\big(q[q_0^*, V^*], q_{0}^*, V^* \big)  \leq J\big(q[q_0, V_0],q_0, V_0\big).
    \end{align*}
    The desired optimality follows (as $q$ is an entropy solution to the local conservation law).     
\end{proof}

We can offer some thoughts about the regularity of the data.
\begin{remark}[\(L^{1}_{\text{loc}}\), \(TV\) bounds]
It might look surprising that no coercivity (or similar properties) of the objective function with respect to the initial datum in some topology is required. This is because we have assumed a uniform \(TV\) bound on \(\mQ\), which gives---due to the given uniform \(L^{\infty}\) bound---a minimizing sequence but also strong convergence in \(C([0,T];L^{p}_{\text{loc}}(\R))\) for all \(p\in[1,\infty)\). A penalization term of the total variation could also be put into the objective function, leading to the same result. However, from the perspective of optimal control, such a penalization is not easy to handle, and one could imagine replacing it with an even stronger norm.
\end{remark}

\begin{remark}[generalizations to other kernels and singular limit convergence for objectives involving the nonlocal solution]
The previous statement can be generalized as follows: If one uses the exponential kernel \cite{CocliteSL2022} or the kernel class of fixed-support (which are monotone as in \cite[Assumption~2]{Keimer2025}), one also has the strong convergence of \(q_{\eta}\) toward the local entropy solution in \(C([0,T];L^{1}_{\text{loc}}(\R))\), enabling the singular limit to be considered in optimal control problems for objective functions of the given class when replacing the nonlocal term \(W_{\eta}[q_{\eta}]\) by \(q_{\eta}\), after switching to~\cref{ass:optimal_control}~\cref{item:3}. 
Not so for the class of convex kernels, as stated in \cref{theo:singular_limit}. Here, the convergence of \(q_{\eta}\) toward the local entropy solution is so far only weak\(*\) in \(L^{\infty}\), so we would need to assume that the objective function is continuous with respect to the named weak\(*\) convergence in \(q_{\eta}\).
\end{remark}
\begin{remark}[existence of optimizers for the local optimal control problem in \cref{def:locp}]
Notice that the singular limit convergence obtained in \cref{theo:singular_limit} also provides the existence of an optimizer of the local optimal control problem in \cref{def:locp}, given that~\cref{ass:optimal_control}~\cref{item:2} (or~\cref{item:3} in the case of the foregoing remark) holds and thus ensures the well-posedness of the optimal control problem.
\end{remark}

\section{Examples and numerical studies}\label{sec:examples_numerics}
To validate our results numerically, we consider several instances of the problem specified in
\begin{definition}[Optimal control problem in the numerical setup]
 \begin{equation}\label{eq:Numpr}
\begin{aligned}
      \min_{q_0 \in \mQ}  & J_i(q_\eta(t, x), q_0) \\
        \mathrm{s.t.~}  \partial_t q_\eta(t, x) + \partial_x \left(q_\eta(t, x) V\left(W_\eta(t, x) \right) \right) & = 0 \text{ in }(0, T) \times \R \\
        q_\eta(0, \cdot) & = q_0
\end{aligned}
\end{equation}
for $(q, q_0) \in L^2\big((0, T); L^2(\R)\big) \times L^2(\R)$ $i \in \lbrace 1, 2 \rbrace$ and
\[
W_\eta(t, x) \coloneqq \int^\infty_x \tfrac{1}{\eta} \exp\left( \tfrac{x-y}{\eta}\right) q(t,y)~\mathrm{d}y
\]
for all $(t, x) \in (0, T) \times \R$. Here, either
\begin{align*}
J_1(q, q_0) & \coloneqq \int_{I} \left \lvert \int_x^\infty \tfrac{1}{\eta}\exp\left(\tfrac{x-y}{\eta} \right)\left( q_\eta(T, y) -  q_d(y) \right)~\dd y \right \rvert^2 \dd x \text{ or }\\
J_2(q, q_0) & \coloneqq \int_{I} \left \lvert q(t, x) - q_d(x) \right\rvert^2 \dd x 
\end{align*}
for all $(q, q_0) \in L^2\big((0, T) \times \R\big) \times L^2(\R)$, which suits~\cref{ass:optimal_control},~\cref{item:1}. 
\end{definition}
The remaining new variables will be clarified in the following.
\begin{assumption}[general assumptions on data]\label{as:gendata}
    We assume that $\eta \in \R_{>0}$, $q_{\max} \in \lbrace 1, 4 \rbrace$, and $T = \tfrac{1}{2}$, $I = [-0.6, 1.6]$, $\Delta x = 0.0025$ $V(x) \coloneqq 1-x$ for all $x \in \R$ and the tracking goal is either one of the following mappings $q_{d, i}:\mathbb{R} \rightarrow \mathbb{R}$ for $i \in \lbrace 1, 2, 3 \rbrace$: We set $q_{d, 1}(x) \coloneqq \chi_{[0, 1]}(x)$ for all $x \in \mathbb{R}$, $q_{d, 2}:= q\left(\tfrac{1}{2}, \cdot \right)$, where $q$ is the solution to 
\[
\partial_t q(t, x) + \partial_x\left(q(t, x)V\left(\tfrac{1}{\eta}\int_x^\infty \exp\left(\tfrac{y-x}{\eta} \right)q(t, y) ~\mathrm{d}y\right)  \right) = 0, \quad q(0, \cdot) = q_0,
\]
with $q(0, x) = \chi_{[0, 1]}(x)$ for all $x \in \mathbb{R}$ and $q_{d, 3}(x) := (1-x)\chi_{[0, 1]}(x)$ for all $x \in \R$.  
\end{assumption}

\subsection{Discretization and discrete adjoint}
To solve~\eqref{eq:Numpr}, it is helpful to know the derivative of the simplified objective functional $q_0 \mapsto J(q[q_0], q_0)$, where $q[q_0]$ is the solution to the nonlocal equation described in the constraints of~\eqref{eq:Numpr} with $q[q_0](0, \cdot) = q_0$. 
However, rather than using the adjoint of the continuous problem and discretizing it afterward, we discretize the problem~\eqref{eq:Numpr} first and then compute an already discrete derivative. This has the advantage of both not requiring us to decide the numerical scheme used for possible adjoint equations and having access to an \emph{exact} gradient. 
\begin{assumption}[assumptions on discretization]\label{as:asd}
We divide the ``enlarged'' interval $\mathcal{I} \coloneqq \left[-0.6-\Delta x, 1.6 + \Delta x \right]$ into an equidistant grid $x_0, \dots, x_P$ for some $P \in \N$ with step length $\Delta x$ as in~\cref{as:gendata} in the following way:
\begin{align*}
-0.6 & =: y_0 < y_0 + \Delta x =: y_1 < \dots < y_{P+1} \coloneqq y_0 + (P+1)\Delta y = 1.6. \\
x_j & := \tfrac{1}{2}\big(y_{j}+y_{j+1}\big) \quad \forall j \in \lbrace 0, ..., P \rbrace
\end{align*}
Moreover, we set 
\[
\gamma_k \coloneqq \int_{k\Delta x}^{(k+1)\Delta x} \tfrac{1}{\eta} \exp\left(\tfrac{-x}{\eta} \right)~\dd x = \exp\left( -\tfrac{k \Delta x}{\eta} \right) - \exp\left(- \tfrac{(k+1) \Delta x }{\eta} \right)
\]
for all $k \in \N_0$. To discretize the time horizon, we fix a time step size $\Delta t \in \R_{>0}$ (to be specified later) and an appropriate $N \in \N$, similarly setting
\[
0 =: t_0 < t_0 + \Delta t =: t_1 < \dots < t_0 + N\Delta t \coloneqq t_N = \tfrac{1}{2}.
\]
The discretization of the governing nonlocal equation is conducted according to the following numerical scheme with Dirichlet boundary conditions~\cite[Definition~4.10]{Keimer2019}:
\begin{align}
    q^0_j &= q_0(x_j)\quad \forall j \in \lbrace 0, \dots, P\rbrace \label{eq:num1}\\
    q^n_0 & = q^n_P = 0 \quad \forall n \in \lbrace 1, \dots, P \rbrace \label{eq:num2}\\
    q^{n+1}_j &= q_j^n + \tfrac{\alpha \Delta t}{2 \Delta x}\big(q_{j-1}^n-2q_j^n + q_{j+1}^n \big) \notag  \\
    & \qquad + \tfrac{\Delta t}{2 \Delta x}\left(q^n_{j-1}V\left(\sum_{k=0}^{P-j+1} \gamma_k q_{j-1+k} \right)  - q^n_{j+1}V\left(\sum_{k=0}^{P-j-1} \gamma_k q_{j+1+k} \right) \right) \label{eq:num3} \\
    & \quad \forall j \in \lbrace 1, \dots, P-1\rbrace,\quad n\in \lbrace 0, \dots, N-1 \rbrace. \label{eq:num4}
\end{align}
Here, inspired by~\cite[Assumption~4.12]{Keimer2019}, we set
\[
\alpha \geq \lVert V \rVert_{L^\infty(0, q_{\max}+1)} + \lVert V' \rVert_{L^\infty(0, q_{\max}+1)} \Delta x \tfrac{1}{\eta} \big(q_{\max}+1\big)
\]
and the CFL condition as
\[
\Delta x \leq \tfrac{2}{2 \alpha + \lVert V' \rVert_{L^\infty(0, q_{\max}+1)} \Delta x \tfrac{1}{\eta} },
\]
where $q_{\max}$ is as in the definition of $\mQ$. We chose $q_{\max}+1$ instead of $q_{\max}$ so that we can ensure the stability of the scheme when operating close to the boundary but outside the admissible set.\\ 
Furthermore, after rearranging, we set $F_{\Delta}: \R^{(P+1)\cdot(N+1)} \times \R^{P+1} \rightarrow \R^{(N+1)\cdot (P+1)}$ such that $F_{\Delta}(q, q_0) =0$ if and only if $q$ adheres to~\cref{eq:num1,eq:num2,eq:num3,eq:num4}.
\end{assumption}

Using this, we consider the following discretized problem.
\begin{definition}[discretized optimization problem including control constraints]\label{def:optd}
We want to solve
\begin{align}
\min_{q_0 \in \R^{P+1}} & \quad G(q[q_0], q_0) \label{eq:constr0}\\
\mathrm{s.t.~} & F_{\Delta}(q[q_0], q_0) = 0\label{eq:constr}\\
& \quad 0\leq (q_0)_j \leq q_{\max} \quad \forall j \in \lbrace 0, \dots, P \rbrace  \label{eq:constr1}.
\end{align}
The discretized objective functional by $G :\R^{(P+1)\cdot(N+1)} \times \R^{P+1} \rightarrow \R$, where $G \in \lbrace J_{\Delta, W}, J_{\Delta, q} \rbrace$ is given by either
\begin{align}
    J_{\Delta, W} & :=  \tfrac{\Delta x}{2} \sum_{j=0}^P \left( \sum_{k=0}^{P-j} \gamma_k q^N_{j+k} - \sum_{k=0}^{P-j} \gamma_k q_d\big(x_{j+k}\big)\right)^2 \text{ or } \label{eq:JW}\\
    J_{\Delta, q} & :=  \tfrac{\Delta x}{2} \sum_{j=0}^P \left( q^N_j -q_d(x_j)\right)^2 \label{eq:Jq}
\end{align}
for all $(q, q_0) \in \R^{(P+1)\cdot(N+1)} \times \R^{P+1}$. Here, $F_{\Delta}$ is as in~\cref{as:asd}.
\end{definition}

For optimization, it is crucial to be able to compute the gradient $\nabla_{q_0} G\big(q[q_0], q_0\big)$ for admissible $q_0$ and $i \in \lbrace 1, 2 \rbrace$, where $q[q_0]$ is the vector in $\R^{(P+1)\cdot(N+1)}$ that arises from~\cref{eq:constr} for some given admissible $q_0$. This can be achieved via the adjoint method.
\begin{theorem}[gradient of the simplified objective function]
    Let $q_0$ be admissible for~\cref{eq:constr0,eq:constr,eq:constr1}, and let $q[q_0]$ denote the vector that is then given by the system~\cref{eq:constr}. We set
    \[
    S_{\ell, m, r, n} \coloneqq \sum_{k = \ell}^m \gamma_k q_{k+r}^n \quad \text{for all $\ell, m, r \in \lbrace 0,\dots, P \rbrace$}.
    \]
    Subsequently, it holds that $\nabla_{q_0}G \big(q[q_0], q_0\big)$ for $G \in \lbrace J_{\Delta, W}, J_{\Delta, q} \rbrace$ as specified in~\cref{def:optd} is given by $\big(p^0_j\big)_{j=0, \dots, P} \in \R^{P+1}$ for $i \in \lbrace W, q \rbrace$, which is iteratively computed by the ``discrete adjoint scheme'' (here $q:= q[q_0]$)
    \begin{align*}
        p^N_j & = \partial_{q_j^N} G(q, q_0) 
        \quad  \forall j \in \lbrace 0, \dots, P\rbrace \\
        p_0^n & = \tfrac{\alpha \Delta t}{2\Delta x}p^{n+1}_1 +\tfrac{\Delta t}{2\Delta x} V\left( S_{0, P, 0, n} \right)p^{n+1}_1 + \tfrac{\Delta t}{2 \Delta x} V'\left( S_{0, P, 0, n} \right)\gamma_0 q^n_0 p^{n+1}_1 \\
        & \forall n \in \lbrace 0, \dots, N-1 \rbrace \\
        p_1^n & = p_1^{n+1} + \tfrac{\alpha \Delta t}{2\Delta x}\big(-2p_1^{n+1} + p_2^{n+1} \big) + \tfrac{\Delta t}{2\Delta x}V\left(S_{0, P-1, 1, n} \right)p_2^{n+1}\\
        & \quad + \tfrac{\Delta t}{2\Delta x} \sum_{\ell = 0}^1 q^n_{\ell} V'\left(S_{0, P-\ell, \ell, n} \right) \gamma_{1-\ell} p^{n+1}_{\ell+1}\\
        &  \forall n \in \lbrace 0, \dots, N-1 \rbrace \\
        p_{P-1}^n & =  p_{P-1}^{n+1} + \tfrac{\alpha \Delta t}{2 \Delta x}\big(p_{P-2}^{n+1}-2p_{P-1}^{n+1}\big) - \tfrac{\Delta t}{2 \Delta x} V\left(S_{0, 1, P-1, n} \right)p_{P-2}^{n+1} \\
        & \quad + \tfrac{ \Delta t}{2 \Delta x}\bigg( \sum_{\ell=0}^{P-2} q_\ell^{n} V'\left(S_{0,P-\ell, \ell, n} \right) \gamma_{\eta, P-1-\ell} p_{\ell+1}^{n+1} \\
        & \qquad -  \sum_{\ell=0}^{P-3} q_{2+\ell}^{n} V'\left( S_{0, P-\ell-2, 2+\ell, n} \right) \gamma_{\eta, P-3-\ell} p_{\ell+1}^{n+1} \bigg) \\
        & \forall n \in \lbrace 0, \dots, N-1 \rbrace \\
        p_{P}^n & =  \tfrac{\alpha \Delta t}{2 \Delta x} p_{P-1}^{n+1} - \tfrac{\Delta t}{2 \Delta x} V\left(S_{0, 0, P, n}\right)p_{P-1}^{n+1} \\
        & \quad + \tfrac{ \Delta t}{2 \Delta x}\bigg( \sum_{\ell=0}^{P-2} q_\ell^{n} V'\left(S_{0, P-\ell, \ell, n} \right) \gamma_{\eta, P-\ell} p_{\ell+1}^{n+1} \\
        & \qquad -  \sum_{\ell=0}^{P-2} q_{2+\ell}^{n} V'\left(S_{0, P-\ell-2, 2+\ell, n}  \right) \gamma_{\eta, P-2-\ell} p_{\ell+1}^{n+1} \bigg) \\
        & \forall n \in \lbrace 0, \dots, N-1 \rbrace \\
        p_j^n & = p^{n+1}_j + \tfrac{\alpha \Delta t}{2 \Delta x} \big(p_{j-1}^{n+1} - 2p_j^{n+1} + p_{j+1}^{n+1} \big) + \tfrac{\Delta t}{2 \Delta x}  V\left( S_{0, P-j, j, n} \right)\big(p_{j+1}^{n+1} - p_{j-1}^{n+1} \big)  \\
        & \quad + \tfrac{\Delta t}{2 \Delta x}  \bigg( \sum_{\ell=0}^j q_\ell^{n} V'\left(S_{0, P-\ell, \ell, n} \right) \gamma_{\eta, j-\ell} p_{\ell+1}^{n+1} \\
        & \qquad -  \sum_{\ell=0}^{j-2} q_{2+\ell}^{n} V'\left(S_{0, P-\ell-2, 2+\ell, n}  \right) \gamma_{\eta, j-2-\ell} p_{\ell+1}^{n+1}\bigg) \\ 
        \\&   \forall n \in \lbrace 0, \dots, N-1 \rbrace,~j \in \lbrace 2, \dots, P-2 \rbrace. & 
    \end{align*}
\end{theorem}
\begin{proof}
    Apply the formulae in~\cite[section~II]{Fikl2016} or~\cite[Section~3]{Giles2003}.
\end{proof}

\subsection{Results}
To solve the discretized problem in~\cref{def:optd}, we apply a projected (remember the box constraint~\cref{eq:constr1}) gradient descent algorithm with Armijo step size rule~\cite[Section 2.3]{bertsekas}. \\
In the following, we determine whether the algorithm yields a reasonable solution and, if so, analyze its behavior as $\eta \rightarrow 0$. The initial guess for gradient descent is set to $0$.
\subsubsection{Solvability}
We are interested in whether the algorithm yields a good approximation of the optimal solution. 
For this, we will consider the two tracking goals $q_{d, 1}$ and $q_{d, 2}$ from~\cref{as:gendata} and $\eta = 0.01$. The results can be observed in~\cref{fig:lr}.
\begin{figure}[h!]
    \centering
    \includegraphics[clip,trim=0 23 3 0 0]{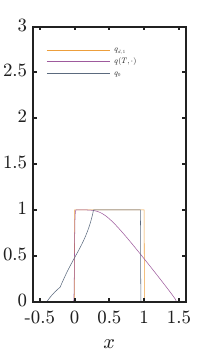}
    \includegraphics[clip,trim=15 23 3 0 0]{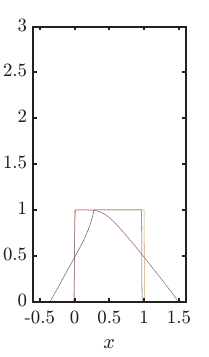}
    \includegraphics[clip,trim=15 23 3 0 0]{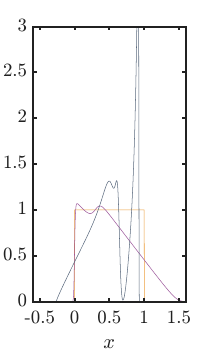}
    \includegraphics[clip,trim=15 23 3 0 0]{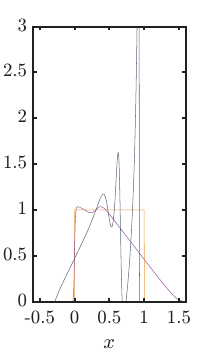} \\
    \includegraphics[clip,trim=0 23 3 0 0]{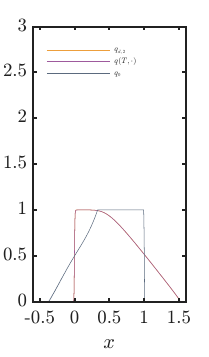}
    \includegraphics[clip,trim=15 23 3 0 0]{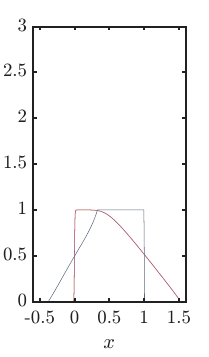}
    \includegraphics[clip,trim=15 23 3 0 0]{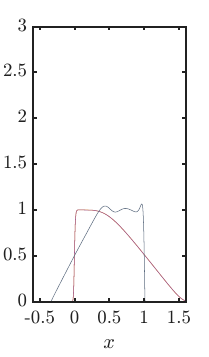}
    \includegraphics[clip,trim=15 23 3 0 0]{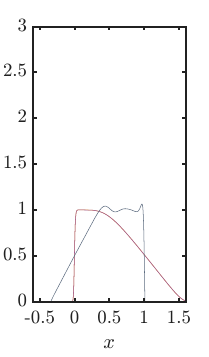}
    \caption{The optimal solution for $\eta = 0.01$ to~\cref{def:optd} $q_0$ found by gradient descent for $q_{d, 1}$ (top row) and $q_{d, 2}$ (bottom row) with $q_{\max} =1$ (first two columns), $J_{\Delta, q}$ as in~\cref{eq:Jq} (first and third column) and $q_{\max} = 4$ (last two columns) and $J_{\Delta, W}$ as in~\cref{eq:JW} (second and fourth column).}
    \label{fig:lr}
\end{figure}
We can see that in the case of $q_{d, 1}$, the desired end datum is not reached very well. This is by design: The smoothing effect on the ``wave front'' by the nonlocal conservation law (see, e.g.,~\cite[Fig.~2]{Friedrich2024}, which is not the same scenario but demonstrates the same effect) makes sure that a characteristic function cannot be achieved after $T = \tfrac{1}{2}$. In the case of $q_{\max} = 4$, the solution is a little closer to the tracking goal, as it is allowed to ``correct higher up''. In the top left picture, we can observe slightly more ``waviness'' in $q_0$. This is likely due to the numerical diffusivity of the PDE solver and ill-conditioning of the problem, caused by the ``artificially'' low bound $q_{\max}$. We hypothesize the nonlocal kernel to have a smoothening effect in the case of $J_{\Delta, W}$, which results in less waves. This can be observed even more in the next scenario.

The other goal $q_{d, 2}$ can be reached very accurately. This is not surprising, as it was chosen as the solution to a nonlocal conservation law. However, the $q_0$ that are obtained do not align with the actual $\chi_{[0, 1]}$ that was chosen to ``create'' $q_{d, 2}$, indicating that due to the closeness to the local optimization problem (cf.~\cref{sec:singular_limit}) and non-injectivity of its control to state map thereof~\cite[p.~8]{Ulbrich2003}, at least numerically, more solutions are possible. 

\subsubsection{Singular limit}
In this subsection, we want to check whether the singular limit result in~\cref{theo:singular_limit} can be validated numerically, that is, whether we can observe numerical convergence of the optimal $q_0$  to a solution to the ``local'' optimal control problem as $\eta \rightarrow 0$. Here, the goal is $q_{d, 3}$ from~\cref{as:gendata}, which is the entropy solution to the \emph{local} scalar conservation law
\[
\partial_t q(t, x) + \partial_x \big(V(q(t, x))q(t, x) \big) = 0, \quad q(0, \cdot) = q_0,
\]
where $q_0(x) \coloneqq \chi_{[0, 1]}(x)$ for $x \in \mathbb{R}$ at time $T = \tfrac{1}{2}$. The results can be seen in~\cref{fig:numsl}.
\begin{figure}[h!]
    \centering
    \includegraphics[clip,trim=0 22 3 0 0]{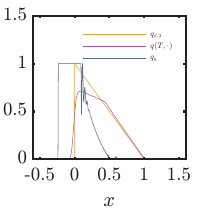}
    \includegraphics[clip,trim=15 22 3 0]{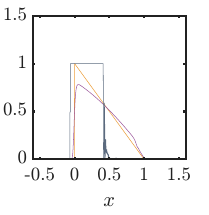}
    \includegraphics[clip,trim=15 22 3 0 0]{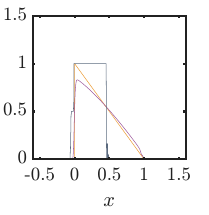}
    \includegraphics[clip,trim=15 22 3 0 0]{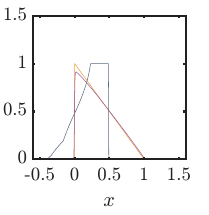}\\
    \includegraphics[clip,trim=0 24 3 0 0]{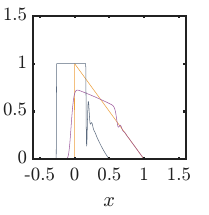}
    \includegraphics[clip,trim=15 24 3 0 0]{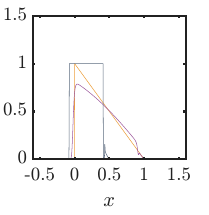}
    \includegraphics[clip,trim=15 24 3 0 0]{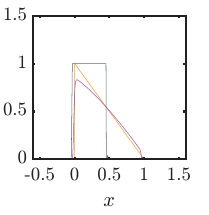}
    \includegraphics[clip,trim=15 24 3 0 0]{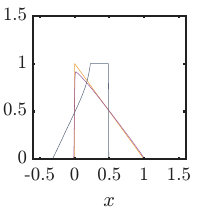}\\
    \includegraphics[clip,trim=0 23 3 0 0]{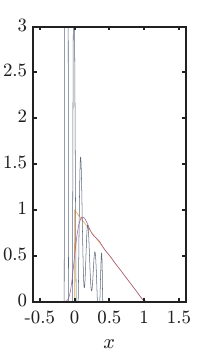}
    \includegraphics[clip,trim=15 23 3 0 0]{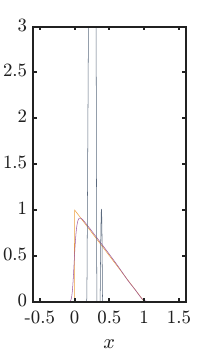}
    \includegraphics[clip,trim=15 23 3 0 0]{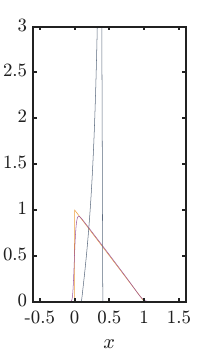}
    \includegraphics[clip,trim=15 23 3 0 0]{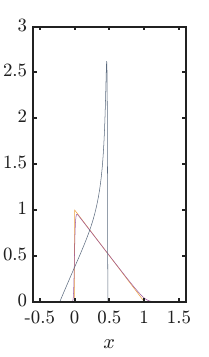}\\
    \includegraphics[clip,trim=0 0 3 0 0]{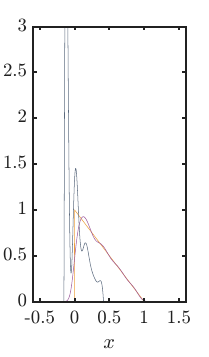}
    \includegraphics[clip,trim=15 0 3 0 0]{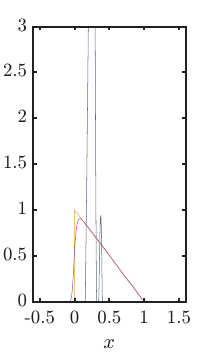}
    \includegraphics[clip,trim=15 0 3 0 0]{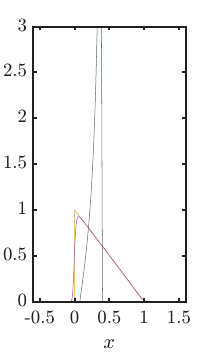}
    \includegraphics[clip,trim=15 0 3 0 0]{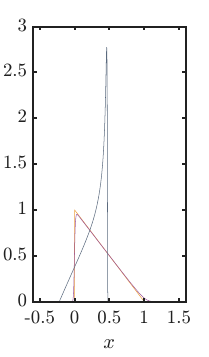}

    \caption{Final time tracking of $q_{d, 3}$ in the sense of~\cref{def:optd} for $q_{\max} = 1$ and $J_{\Delta, q}$ as in~\cref{eq:Jq} (first row), $q_{\max} = 1$ and $J_{\Delta, W}$ as in~\cref{eq:JW} (second row), $q_{\max} = 4$ and $J_{\Delta, q}$ (third row) and $q_{\max} = 4$ and $J_{\Delta W}$(bottom row). Moreover, we set $\eta \in \lbrace 0.01, 0.05, 0., 0.5 \rbrace$ (each row from left to right).}
    \label{fig:numsl}
\end{figure}

We can not always observe ``numerical convergence'' of the $q_0$. as $\eta \rightarrow 0$. This depends, as mentioned in the prior subsection, on the conditioning of the problem. When $q_{\max} = 4$, one can see better, how the solution structurally morphs into what can be observed for $\eta  = 0.01$. Moreover, we can see that the choice of $J_{\Delta, W}$ usually leads to less oscillations in the optimal solution found by the solver. For $q_{\max} = 4$, the goal is hit particularly well and one can see that decreasing $\eta$ indeed pushes $q(T, \cdot)$ closer to $q_{d, 3}$.\\
Moreover, the solution found is not the initial datum that was chosen to obtain $q_{d, 3}$. This again points to the aforementioned noninjectivity of the control to state map of the local problem. 


\section{Open problems}\label{sec:open_problems}

Currently, the singular limit convergence of the optimal control problem requires some compactness estimates, which rely on uniform \(TV\) estimates on the admissible set.
However, one could work with the nonlocal version of Ole\u{\i}nik estimates for the nonlocal operator, as was established in \cite{keimeroleinik}. The advantage would be that no \(TV\) bound on the admissible set of the initial datum would be required, but the nonlocal operator would satisfy a one-sided Lipschitz (OSL) condition in an arbitrarily small time, resulting in compactness. The problem with this approach is that for the OSL condition to hold, the nonlinear \(V\) and the initial datum would need to be restricted further, so unlike classical Ole\u{\i}nik estimates for scalar local conservation laws \cite{oleinik}, this approach might not be general enough.

In a recent preprint \cite{coclite2025singular}, it was proven via compensated compactness that the nonlocal solution converges to the local entropy solution without requiring the additional \(TV\) bound on the initial datum. This could strengthen the results obtained in this contribution, additionally yielding a convergence order of the nonlocal-to-local limit under these weaker assumptions, similar to what has been established in \cite[Theorem~1.3]{Colombo2023}.

At the moment, we rely for the numerical work on an adapted Lax--Friedrich method (compare \cref{sec:examples_numerics}). However, other methods, particularly those that are robust when changing the nonlocal reach \(\eta\) and letting it pass to zero \cite{Denitti2025asymptotically}, are important to consider in further studies.

Eventually, it might be worth knowing about the speed of convergence with respect to the nonlocal optimizers. A key result for convergence speed to the local entropy solution without the optimization framework is \cite{Colombo2023}; however, this is not enough to tell us anything about the speed of convergence with respect to the optimization.

\section*{Acknowledgments}
The research was funded by the Deutsche Forschungsgemeinschaft
(DFG, German Research Foundation) under Project-ID 416229255 (CRC
1411) and Project-ID 547096773.

\bibliographystyle{siamplain}
\bibliography{obstacle}

\end{document}

%% file: ex_shared.tex
\usepackage{lipsum}
\usepackage{amsfonts}
\usepackage{graphicx}
\usepackage{epstopdf}
\usepackage{algorithmic}
\usepackage{mathtools,amssymb}
\usepackage{todonotes}
\usepackage{enumitem}

\ifpdf
  \DeclareGraphicsExtensions{.eps,.pdf,.png,.jpg}
\else
  \DeclareGraphicsExtensions{.eps}
\fi

\newcommand{\eps}{\varepsilon}
\newcommand{\R}{\mathbb{R}}
\newcommand{\N}{\mathbb{N}}

\newcommand{\mV}{\mathcal{V}}
\newcommand{\mQ}{\mathcal{Q}}
\newcommand{\mC}{\mathcal{C}}
\newcommand{\dd}{\,\mathrm{d}}
\DeclareMathOperator{\sgn}{sgn}
\DeclareMathOperator{\essinf}{ess-\inf}
\newsiamremark{remark}{Remark}
\newsiamremark{hypothesis}{Hypothesis}
\crefname{hypothesis}{Hypothesis}{Hypotheses}
\newsiamremark{assumption}{Assumption}
\crefname{assumption}{Assumption}{Assumptions}
\newsiamthm{claim}{Claim}
\newsiamremark{fact}{Fact}
\crefname{fact}{Fact}{Facts}

\headers{Optimal control of nonlocal conservation laws}{A. Keimer, L. Pflug, J. Rodestock}

\title{An Example Article\thanks{Submitted to the editors DATE.
\funding{This work was funded by the Fog Research Institute under contract no.~FRI-454.}}}

\author{Alexander Keimer\thanks{Institute of Mathematics,
University of Rostock, 18055 Rostock, Germany
  (\email{alexander.keimer@uni-rostock.de})}
\and Lukas Pflug\thanks{Department of Mathematics, Friedrich-Alexander-University Erlangen-Nürnberg, 91058 Erlangen, Germany (\email{lukas.pflug@fau.de})}
\and Jakob Rodestock\thanks{Department of Mathematics, Friedrich-Alexander-University Erlangen-Nürnberg, 91058 Erlangen, Germany (\email{jakob.rodestock@fau.de})}}

\usepackage{amsopn}
